\documentclass{article}

\usepackage{fullpage}
\usepackage{amssymb,dsfont,color}
\usepackage{amsthm}
\usepackage{soul}
\usepackage{amsmath}
\usepackage{mathrsfs}
\usepackage{enumerate}
\usepackage{mathtools}

\DeclarePairedDelimiter\floor{\lfloor}{\rfloor}
\newtheorem{theorem}{Theorem}

\newtheorem{proposition}[theorem]{Proposition}

\newtheorem{assumption}[theorem]{Assumption}
\newtheorem{remark}[theorem]{Remark}
\newtheorem{lemma}[theorem]{Lemma}

\newtheorem{definition}[theorem]{Definition}
\newtheorem{example}[theorem]{Example}
\def\ds{\displaystyle}
\def\s{\sigma}
\def\R{\mathbb{R}}
\def\e{\varepsilon}
\def\S{\mathcal{S}}

\def\M{\mathcal{M}}
\def\N{\mathcal{N}}
\def\C{\mathcal{C}}

\def\s{\sigma}
\def\dist{d}

\definecolor{orange}{rgb}{0.99,0.69,0.07}

\begin{document}

\title{
Upper and lower bounds for the maximal Lyapunov exponent of singularly perturbed linear switching systems}
 
  \author{
 Yacine Chitour\thanks{Laboratoire des Signaux et Syst\`emes (L2S),  Universit\'e Paris-Saclay, CNRS, CentraleSup\'elec,  Universit\'e Paris-Saclay, Gif-sur-Yvette, France, {\tt yacine.chitour@l2s.centralesupelec.fr}}, 
Ihab Haidar\thanks{Quartz EA 7393, ENSEA, 95000, Cergy, France, {\tt ihab.haidar@ensea.fr}},
 Paolo Mason\thanks{CNRS \& Laboratoire des Signaux et Syst\`emes (L2S),  Universit\'e Paris-Saclay, CNRS, CentraleSup\'elec, Gif-sur-Yvette, France, {\tt paolo.mason@l2s.centralesupelec.fr}}, and Mario Sigalotti\thanks{
 Sorbonne Université, Inria, CNRS, Laboratoire Jacques-Louis Lions (LJLL), Paris, France, {\tt mario.sigalotti@inria.fr}}}

\maketitle
\begin{abstract}
In this paper we consider the problem of determining the stability properties, and in particular assessing the exponential stability, of a singularly perturbed linear switching system. One of the challenges of this problem arises from the intricate interplay between the small parameter of singular perturbation and the rate of switching, as both tend to zero. 
Our approach consists in characterizing 
suitable auxiliary linear systems
that provide lower and upper bounds for the asymptotics of the maximal Lyapunov exponent of the  linear switching system
as the  parameter of the singular perturbation tends to zero. 
\end{abstract}


\section{Introduction}
 
We consider in this paper a two-time-scales linear switching  system, 
that is, a linear switching  system for which some variables evolve on a much faster rate than the others.
This class of systems appears in several industrial and engineering applications (see, e.g., \cite{4342105,Malloci2009,9199288}) where simplified models can be formulated by neglecting the effects of fast variables on the overall system. From control point of view, this allows to design a controller based on a reduced order model. However, the design based on a simplified model may not guarantee the stability of the overall system. To avoid this problem, a well-established framework developed in the mathematical and control community is that of singular perturbations \cite{KKO}. The singular perturbation theory allows the separation between slow and fast variables where different controllers for different time-scale variables can be designed in order to lead the overall system to its desired performance. 

In mathematical terms, we study the behavior of
\begin{align*}
\dot x(t)&=A(t)x(t)+B(t)y(t),\\
\e\dot y(t)&=C(t)x(t)+D(t)y(t),
\end{align*}
where $\e$ denotes a small positive parameter and $A,B,C,D$ are matrix-valued 
{signals}
undergoing arbitrary switching within a prescribed bounded range. 
We deal, therefore, with a {1-parameter} 
family 
of linear switching systems $\Sigma:\e\mapsto \Sigma_\e$.  
One of the main issues for such families of systems consists in understanding the time-asymptotic behavior of $\Sigma_\e$ as $t\to+\infty$ in the regime where $\e$ is small.
For instance, by saying that \emph{$\Sigma$ is exponentially stable} we refer to the fact that for every $\e>0$ sufficiently small, the corresponding linear switching system $\Sigma_\e$ is exponentially stable. 
More refined notions, accounting for the uniform exponential behavior with respect to $\e$, are proposed in the paper (cf.~Definition~\ref{0-GES def}). 
Few stability criteria  for singularly perturbed switching systems in the regime $\e\sim0$ have been obtained
in the literature: among them, let us mention \cite{Malloci2009CDC}, where conditions are obtained based on the existence of a common quadratic Lyapunov function, \cite{Hachemi2011} characterizing the stability in dimension two
based on the corresponding criteria in the non-singularly-perturbed case \cite{Balde2009,Boscain2005}, and \cite{Alwan2008}, where  stability 
for time-delay singularly perturbed switching systems 
is based on  dwell-time criteria. 

A major mathematical difficulty 
relies on the fact that we are interested in characterizing a doubly asymptotic regime, 
where the order in which the limits are taken is crucial. 
Indeed, a limit as $\e\to 0$ of the considered dynamics is well known to be given, through the Tikhonov decomposition, by
\begin{equation}\label{slow-intro}
\dot x(t)=\left(A(t)-B(t)D(t)^{-1}C(t)\right)x(t),
\end{equation}
with $y(t)=-D(t)^{-1}C(t)x(t)$. 
 Heuristically, such a decomposition is 
obtained by saying that $y$ tends instantaneously to $-D(t)^{-1}C(t)x(t)$ (the equilibrium of the equation for $y$ when $x\equiv x(t)$). 
However, as it has already been observed in the literature,  
 the switching system \eqref{slow-intro} may be exponentially stable even when, for every $\e>0$, $\Sigma_\e$ is unstable \cite{Malloci2009CDC}. 
Notice that, in order to justify  the Tikhonov decomposition, we are assuming here that each $D(t)$ is a Hurwitz matrix. Actually, a standing assumption in this work is that  
the fast dynamics are exponentially stable, i.e., the trajectories of 
\begin{equation}\label{eq:fast-dym}
\dot y(t)=D(t)y(t)
\end{equation}
converge to the origin 
with a uniform exponential
rate. It should be noticed that this condition is necessary for 
the exponential stability of $\Sigma_\e$ in the regime $\e\sim 0$, since if 
\eqref{eq:fast-dym} admits trajectories diverging exponentially, then there is no hope for $\Sigma_\e$ to be exponentially  stable for $\e$ small (a precise mathematical statement of this fact is proved in Proposition~\ref{prop:toto1}).
In the limit situation where \eqref{eq:fast-dym} is stable but not exponentially stable, we can still have exponential stability of $\Sigma_\e$ for $\e\sim 0$ \cite[Section V.A]{Hachemi2011}, but this is a rather degenerate 
situation that we do not consider here.

The goal of this paper is not only to give
necessary or sufficient conditions ensuring that $\Sigma$ has a certain time-asymptotic behavior in the regime $\e\sim 0$, 
but also to prescribe upper and lower bounds 
on the limit as $\e\to 0$
of the maximal Lyapunov exponent of $\Sigma_\e$. We recall that the maximal Lyapunov exponent of a linear switching system is the largest asymptotic exponential rate as the time goes to $+\infty$ among all trajectories of the system. 
Necessary or sufficient conditions for stability then follow as particular cases: indeed, a positive lower bound ensures that, for every $\e$ small enough, system $\Sigma_\e$ is unstable, while a negative upper bound guarantees that $\Sigma_\e$ is exponentially stable for all $\e$ in a right-neighborhood of zero. 

The bounds on the limit as $\e\to 0$
of the maximal Lyapunov exponent of $\Sigma_\e$ are obtained by identifying 
suitable auxiliary switching systems for the variable $x$ (with a single time scale) that are either sub- or super-approximations of the asymptotic dynamics of $\Sigma_\e$ as $\e\to 0$, in the following sense. A sub-approximation $\bar \Sigma$ of $\Sigma$ is a  system such that each of its  trajectories  can be approximated arbitrarily well, as $\e\to 0$, by the $x$-component of a trajectory of $\Sigma_\e$. Conversely, a super-approximation $\hat \Sigma$ of $\Sigma$ is 
a system such that the $x$-component of each trajectory of $\Sigma_\e$ 
can be approximated arbitrarily well, as $\e\to 0$, by a trajectory  of $\hat \Sigma$.

The first and rather natural choice for $\bar \Sigma$ is system \eqref{slow-intro} since it corresponds to the singular perturbation approach, see for instance \cite{kokotovic84}.  
The heuristic explanation is that, if switching occurs ``slowly'' with respect to the time-scale $1/\e$, then the  variable $y(t)$  converges fast enough to $-D(t)^{-1}C(t)x(t)$, and the transient phase is too short to affect the dynamics of $x(t)$. 
As for a choice of super-approximation $\hat \Sigma$, the idea is to consider the switching parameter as evolving on a time scale possibly much faster than $1/\e$. Hence the dynamics of $y$ can be seen, on a short time interval, as a switching system with affine vector fields $y\mapsto C(t)\bar x+D(t)y$, where $\bar x$ is fixed. This kind of switching systems have been studied, for instance, in \cite{DellaRossa2022,Nilsson2013}. The exponential convergence of the solutions of \eqref{eq:fast-dym} 
implies that the trajectories of such an  affine switching system converge exponentially towards a compact set $K(\bar x)$ \cite{DellaRossa2022}. A natural choice for $\hat \Sigma$ is then the 
{differential inclusion} 
\begin{equation*}
\dot x(t)\in A(t)x(t)+B(t)K(x(t)).
\end{equation*}
We show that, indeed, this system allows to identify an upper bound for the limit of the maximal Lyapunov exponent of $\Sigma_\e$ as $\e\to 0$. 

As mentioned in the previous paragraph, the two choices of $\bar \Sigma$ and $\hat \Sigma$ just described correspond to  signals switching at a rate much slower or much faster than $1/\e$. 
We also consider a third regime, namely, the situation in which the switching occurs at rate exactly $1/\e$. This corresponds to considering a  signal $t\mapsto (A(t),B(t),C(t),D(t))=:\sigma(t)$ defined on an interval $[0,T]$ and all its reparameterizations $\sigma_\e:t\mapsto \sigma(t/\e)$ defined on $[0,\e T]$. 
The flow of $\Sigma_\e$ corresponding to $\sigma_\e$, evaluated at time $\e T$,
has an effect of order $O(\e)$ on the coordinate $x$ (since the velocities are of order 1 and the length of the time interval is of order $\e$) and of order $O(1)$ on the coordinate $y$. 
Suitable computations show that the evolution of $x$ can actually be described as
\[x\mapsto x+\e T\Lambda(T,\sigma)x+O(\e^2),\]
where the term $\Lambda(T,\sigma)$ does not depend on the initial condition of the variable $y$. 
We then propose another choice of sub-approximation of $\Sigma$, denoted $\check \Sigma$, which coincides with the switching system having as possible modes all the matrices of the type $\Lambda(T,\sigma)$. 
We provide an explicit expression for such matrices and we prove that, indeed, each 
trajectory of $\check \Sigma$ can be approximated arbitrarily well, as $\e\to 0$, by the $x$-component of a trajectory of $\Sigma_\e$.
Moreover, we show that the modes of 
$\bar \Sigma$ are also modes of $\check\Sigma$, implying that the 
maximal Lyapunov exponents of $\bar \Sigma$ is upper bounded by that of $\check\Sigma$. 
The main result of our paper, therefore, 
consists
in 
providing an interval containing all limits points of the maximal Lyapunov exponents of $\Sigma_\e$ as $\e$ goes to zero: the upper and lower bounds are given by the maximal Lyapunov exponents of $\hat \Sigma$ and $\check\Sigma$, respectively (see Theorem~\ref{thm:main}). 

The paper is organized as follows:
In Section~\ref{sec:statement} we present precise definitions of the systems $\bar \Sigma$, $\check \Sigma$, $\hat \Sigma$, and of their maximal Lyapunov exponents. We also provide the  statement of the main result, Theorem~\ref{thm:main}, whose proof is given in the reminder of the paper. In particular, in Section~\ref{sec:bar} we compare the 
maximal Lyapunov exponents of $\bar \Sigma$ and the 
asymptotic {values}
as $\e$ goes to zero of the maximal Lyapunov exponent of $\Sigma_\e$ (Proposition~\ref{epsbehavior}). The comparison with the maximal Lyapunov exponent of $\check\Sigma$ is the subject of 
Section~\ref{czech} (Proposition~\ref{check}), while Section~\ref{sec:hat} discusses the comparison with the maximal Lyapunov exponent of $\hat\Sigma$.

\subsection{Notations}
By $\R$ we denote the set of real numbers, by $|\cdot|$ the Euclidean norm of a real vector, and by $\|\cdot\|$ the induced matrix norm.
We use $\R_+$ to denote the set of non-negative real numbers and we write $\floor{x}$ to denote the 
evaluation of the 
floor function at a real number $x$. 
By $M_n(\R)$ we denote the set of $n\times n$ real matrices. The $n\times n$ identity matrix is denoted by $I_{n}$. We use $\rho(M)$ to denote the spectral radius of a matrix $M\in M_n(\R)$, defined as the largest modulus among the eigenvalues of $M$.  
By $B_r(x)$ 
  we denote the closed ball of radius $r>0$ and center $x\in\R^n$. The Hausdorff distance between two nonempty subsets $X$ and $Y$ of $\R^n$ is the quantity defined by 
\[d_H(X,Y)=\max\left\{\sup_{x\in X}d(x,Y), \sup_{y\in Y}d(y,X)\right\},\] 
where $d(x,Y)=\ds\inf_{y\in Y}|x-y|$ and $d(y,X)=\ds\inf_{x\in X}|x-y|$. Given a subset $\N$ of $M_n(\R)$, 
we denote
by $\S_{\N}$  the set of all measurable functions from $\R_+$ to 
 $\N$. 

 \section{Problem statement and main results}\label{sec:statement}

\subsection{Stability notions for singularly perturbed  linear switching  systems}

Fix $n,m\in \mathbb{N}$ and a compact set of matrices $\mathcal{M}\subset M_{n+m}(\R)$.
Let 
$\Sigma=(\Sigma_\varepsilon)_{\varepsilon>0}$ be the family of  linear switching systems
\begin{equation*}
\Sigma_\varepsilon: \quad  \begin{array}{rll}
\quad \dot x(t)&=&A(t)x(t)+B(t)y(t),\\
\varepsilon \dot y(t)&=&C(t)x(t)+D(t)y(t),
\end{array}
\end{equation*}   
where $\e$ denotes a small positive parameter and 
\begin{equation*}
t\mapsto \begin{pmatrix}
A(t)&B(t)\\
C(t) &  D(t)
\end{pmatrix}
\end{equation*}
is an arbitrary element of 
the set 
$\S_{\M}$ of measurable functions from $\R_+$ to $\mathcal{M}$.

For a given $\e>0$, 
the usual stability notions, recalled in the following definition, apply to the 
linear switching system
$\Sigma_\e$. \\

\begin{definition} \label{def:LSS-lyapunov} 
Let $d\in \mathbb{N}$ and $\N$ be a bounded subset of $M_d(\R)$. Consider the linear switching system 
\begin{equation}\label{general}
    \Sigma_{\N}:\quad \dot x(t)=N(t)x(t),\qquad N\in \S_{\N},
\end{equation} 
and denote by $\Phi_{N}(t,0)$ the flow from time $0$ to time $t$ of $\Sigma_{\N}$ associated with the switching signal $N$. 
Then $\Sigma_{\N}$
is said to be
\begin{enumerate}
\item
\emph{exponentially  stable} (ES, for short) if there exist $C>0$ and $\delta>0$ such that 
\[\|\Phi_{N}(t,0)\|\le C e^{-\delta t},\qquad \forall\,t\ge 0, \forall\, N\in \S_{\N};\]
\item
\emph{exponentially  unstable} (EU, for short) if there exist $C>0$, $\delta>0$, and a nonzero trajectory $t\mapsto x(t)$ of $\Sigma_{\N}$ such that
\[|x(t)|\ge C e^{\delta t}|x(0)|,\qquad \forall t\ge 0.\]
\end{enumerate}
The \emph{maximal Lyapunov exponent 
of $\Sigma_{\N}$} is defined as
\begin{equation*}
    \lambda(\Sigma_{\N})=\limsup_{t\to+\infty}\frac{1}{t}\sup_{N\in\S_{\N}}\log\|\Phi_{N}(t,0)\|.
\end{equation*} 
\end{definition}

\begin{remark} 
Notice that, for every $t>0$ and every $N\in \S_{\N}$, by Gelfand's formula,
\[
\rho(\Phi_{N}(t,0))=\lim_{k\to+\infty}\|\Phi_{N}(t,0)^k\|^{\frac1k}=\lim_{k\to+\infty}\|\Phi_{\tilde N}(kt,0)\|^{\frac1k}
\]
where $\tilde N$ denotes the $t$-periodic signal obtained by periodization of $N|_{[0,t]}$.
As a consequence, 
\begin{equation}\label{eq:gelfand}
\rho(\Phi_{N}(t,0))\le e^{t \lambda(\Sigma_\N)}.
\end{equation}
Actually, it  is well known (see, e.g., \cite{WIRTH200217}) that an equivalent definition of the maximal Lyapunov exponent is
\begin{equation}\label{eq:JSR}
\lambda(\Sigma_{\N})=\limsup_{t\to+\infty}\frac{1}{t}\sup_{N\in\S_{\N}}\log\rho(\Phi_{N}(t,0)).
\end{equation}

\end{remark}

\begin{remark} \label{rem:pseudo-fenichel}
It follows from the definitions that $\Sigma_{\N}$  is ES if and only if $\lambda(\Sigma_{\N})<0$. Also, by~\eqref{eq:JSR}, we have that $\Sigma_{\N}$ is EU  if and only if $\lambda(\Sigma_{\N})>0$.
Moreover, the properties of exponential stability/instability and the value of $\lambda(\Sigma_{\N})$ keep unchanged if we replace $\S_{\N}$ by the class of piecewise-constant functions from $\R_+$ to $\N$. 
\end{remark}


In the following we 
write 
$\Phi^\e_{M}(t,0)$ to denote the flow from time $0$ to time $t$ of $\Sigma_\e$ associated with a signal $M\in \S_{\M}$ and 
we 
introduce the following stability notions for the 1-parameter family $\Sigma$ of switching systems. \\

\begin{definition}\label{0-GES def}
We say that the 
1-parameter
family 
of linear switching systems $\Sigma:\e\mapsto \Sigma_\e$ is
\begin{enumerate}
\item  $\e$-\emph{uniformly exponentially stable} ($\e$-ES, for short) if there exist  
$\e^{\star}>0$, $C>0$, and $\delta>0$ such that
\begin{equation}\label{eq-ES-unif}
\|\Phi^\e_{M}(t,0)\|\leq C e^{-\delta t},\,\forall t\geq 0, \,\forall M\in \S_{\M}, \forall \e\in (0,\e^{\star});
\end{equation}
\item $\e$-\emph{uniformly exponentially unstable} ($\e$-EU, for short) if there exist $\delta>0$, $C>0$, and $\e^{\star}>0$ such that for every $\e\in (0,\e^{\star})$ there exist $M\in\S_{\M}$ and $z_0\neq 0$  for which
\begin{equation*}
|\Phi_{M}^\e(t,0)z_0|\geq C e^{\delta t}|z_0|,\qquad \forall\,t\geq 0.
\end{equation*}
\end{enumerate}
\end{definition}

\medskip

\begin{remark}
The $\e$-uniform exponential stability of $\Sigma$ is stronger than the property of global uniform asymptotic stability  introduced in \cite{Hachemi2011}, since it not only guarantees that $\lambda(\Sigma_\e)<0$ for every $\e$ small enough, but also that the stability is uniform with respect to $\e$, i.e.,   $\lambda(\Sigma_\e)$ is smaller than the negative number $-\delta$, independent of $\e$ small, and the constant $C$ appearing in \eqref{eq-ES-unif} is independent of $\e$ small. 
Similar considerations concern {the} $\e$-uniform exponential instability.
\end{remark}

\subsection{Auxiliary switching systems and main result}

Our stability analysis of $\Sigma$ relies on the comparison with  
several auxiliary systems having a single time scale.  

Let us first discuss the system $\Sigma_D$ corresponding to the evolution of the fast variable $y$ with $x=0$. More precisely,
let 
\[\M_D=\{D\mid  (\begin{smallmatrix}A&B\\C&D\end{smallmatrix})
\in \M\}\]
and consider the linear switching system
\[
    \Sigma_{D}:\quad \dot y(t)=D(t)y(t),\qquad D\in \S_{\M_D}.\]
We first provide a useful result for the asymptotic behavior of 
$\Sigma$.\\

\begin{proposition}\label{prop:toto1}
With the notations above, it holds that
\[\lim_{\varepsilon\to 0}\varepsilon\lambda(\Sigma_\varepsilon)=\max\{
0,\lambda(\Sigma_D)
\}
.\]
\end{proposition}
\begin{proof}Consider the family $(\Sigma_{\N_\e})_{\e\geq 0}$ of linear switching  systems, where, for every $\e\geq 0$, the compact set of matrices $\N_\e\subset M_{n+m}(\R)$ is defined by 
\[ \N_\e = \left\{\begin{pmatrix}\e A &\e B\\
C &  D
\end{pmatrix}\,|\,  \begin{pmatrix} A &B\\
C &  D
\end{pmatrix}\in \M\right\}.
\]
Then one notices that, for every $\e>0$, the time rescaling $t\mapsto \e t$ yields 
\[\
\phi_{N_\varepsilon}(t,0)=\phi^{\varepsilon}_{M}(\varepsilon t,0),
\]
where $M$ is an arbitrary signal in $\M$ and the signal $N_\e\in\N_\e$ is defined as 
\[N_\varepsilon(t)=\begin{pmatrix}
\varepsilon A(t)&\varepsilon B(t)\\
C(t) &  D(t)
\end{pmatrix},
\quad 
t\geq 0.
\]This implies at once that, for every $\e$, 
$\varepsilon\lambda(\Sigma_\varepsilon)=\lambda(\Sigma_{\N_\e})$.
Next notice that the 
compact sets $
\N_\e
$ converge 
to $\N_0$ 
as $\e$ goes to zero
for the Hausdorff topology defined on the family of compact subsets of $M_{n+m}(\R)$.
One gets the conclusion
using
Lemma 3.5 and Equation~(19)
from \cite
{WIRTH200217}.
\end{proof}
The above result yields, in the particular case where $\lambda(\Sigma_D)>0$, that $\lambda(\Sigma_\e)\ge
\mu/\e$
for $\mu\in (0,\lambda(\Sigma_D))$ and $\e>0$ small enough, hence 
$\Sigma_\e$ is unstable with trajectories diverging at an arbitrarily large exponential rate  for $\e$ small.
This motivates the following working assumption for the rest of the paper.\\

\begin{assumption}\label{UGES-fast}
The switching system $\Sigma_D$ is ES, that is, $\lambda(\Sigma_D)<0$.
\end{assumption}
 In particular, all matrices in $\M_D$ are Hurwitz (and then invertible). 
%
We also  introduce the  set  
\begin{equation*}
\bar{\mathcal{M}}:=\{A-BD^{-1}C\mid 
(\begin{smallmatrix}A&B\\C&D\end{smallmatrix})
\in\mathcal{M}\}\subset M_{n}(\R),
\end{equation*}
which collects the modes of the switching system
\begin{align*}
\bar\Sigma:  \dot {\bar x}(t)&=\left(A(t)-B(t)D(t)^{-1}C(t)\right)\bar x(t),\qquad (\begin{smallmatrix}A&B\\C&D\end{smallmatrix})\in \S_{\M}.
\end{align*}

As described in the introduction, we also consider as auxiliary system another switching system, denoted by $\check{\Sigma}$,
with a larger set of modes than $\bar \Sigma$, and which
could be {thought} of as the 
slow dynamics corresponding to 
signals whose
switching occurs at rate exactly $1/\e$.
In order to define $\check{\Sigma}$,  let us associate  
with every $T>0$ and every $\sigma=(\begin{smallmatrix}A&B\\C&D\end{smallmatrix})\in \S_{\M}$  the matrices
\begin{align*}
\Lambda_0(T,\sigma)&=\int_{0}^{T}\Phi_{D}(T,s)C(s)\ ds,\\ 
 \Lambda_1(T,\sigma)&=\int_{0}^{T}\left(A(s) +B(s)\Lambda_0(s,\s)\right)\ ds, \\
\Lambda_2(T,\sigma)&=\int_{0}^{T}B(s)\Phi_{D}(s,0)\ ds,
 \end{align*}
and
\[\Lambda(T,\s)=\frac{\Lambda_1(T,\s)+\Lambda_2(T,\s)\left(I_m-\Phi_D(T,0)\right)^{-1}\Lambda_0(T,\s)}T.\]
As proved in Lemma~\ref{paolo-lemma} given in Section~\ref{czech}, the set
\[\mathcal{\check M}:=\{\Lambda(T,\s)\mid T>0,\;\s\in\S_\M\}\]
is bounded in $M_n(\R)$. 
The 
switching system $\check{\Sigma}$
is then defined as
\begin{align}\label{new limit system}
\check\Sigma: \quad \dot {\check x}(t)=M(t)
\check x(t), \quad M\in \S_{\check \M}.
\end{align}

In the case of autonomous systems, i.e., in the absence of switching,
singular perturbation theory~\cite{KOKOTOVIC1976123}
guarantees that
the exponential stability of $\Sigma$
is completely characterized by the associated reduced dynamics~$\bar\Sigma$. 
In the switching case, however, 
it is well known that the stability of $\bar\Sigma$ is not sufficient to deduce the stability of the perturbed switching system $\Sigma_\e$ for $\e\sim 0$~\cite{Malloci2009CDC}.
The main difference
is that, fixed $x\in \R^n$, in the switching case  the $\omega$-limit set $K(x)$ 
of the \emph{fast dynamics}
 \begin{align}\label{fast}
\Sigma_x: \quad \dot y(t)&= D(t)y(t)+C(t) x,\qquad (\begin{smallmatrix}A&B\\C&D\end{smallmatrix})\in \S_{\M}, 
\end{align}
is in general not reduced to 
a single equilibrium point \cite{DellaRossa2022,Nilsson2013}.
The precise definition of the set $K(x)$ is the closure of the union of the $\omega$-limit sets of all trajectories of $\Sigma_x$ starting from the initial condition $y(0)=0$. 
We will prove in Proposition~\ref{forward invariance} and Lemma~\ref{lem:KLiphom}
that $K(x)$ is compact and the set-valued map $x\mapsto K(x)$  is  homogeneous of degree one and globally Lipschitz continuous for the Hausdorff distance.

Let us use the set-valued map $K$ to introduce a last  auxiliary system, which is not necessarily a switching system.
Namely, let us set
\begin{equation*}
\hat{\M}(x):=\{Ax+B y\mid 
(\begin{smallmatrix}A&B\\C&D\end{smallmatrix})\in\mathcal{M},\;y\in K(x)\},\quad x\in \R^n,
\end{equation*}
and consider the differential inclusion
 \begin{align*}
\hat\Sigma: \quad \dot x(t)&\in \hat{\M}(x(t)). \end{align*}
Notice that $\hat{\M}(x)$ is compact for every $x\in\R^n$, since both $\M$ and $K(x)$ are compact. 
We recall that a solution to $\hat\Sigma$  (in the sense of Filippov) is an absolutely continuous map $\R_+\ni t\mapsto x(t)\in \R^n$ such that $\dot x(t)\in \hat{M}(x(t))$ for almost all $t$. The existence of solutions to $\hat{\Sigma}$ for every initial condition is a consequence of 
classical results on differential inclusions
(see, e.g., \cite[Chapter 2]{aubin-cellina}). 

In analogy {with}
Definition~\ref{def:LSS-lyapunov}, we say that 
$\hat\Sigma$ is ES (for \emph{exponentially stable}) if there exist $C>0$ and $\delta>0$ such that every trajectory $x(\cdot)$ of $\hat\Sigma$ satisfies
\[|x(t)|\le C e^{-\delta t}|x(0)|,\qquad \forall t\ge 0.\]
The notion of exponential instability is completely analogous to 
the one given for linear switching systems. As for the Lyapunov exponent of $\hat\Sigma$, it can be defined as 
\begin{equation*}
   \lambda(\hat\Sigma):=\limsup_{t\to+\infty}\frac{1}{t}\sup\log|x(t)|,
\end{equation*} 
where the sup is taken over all trajectories $x(\cdot)$  of~$\hat\Sigma$  with $|x(0)|=1$.
Notice that the restriction to initial conditions with unit norm is justified by the fact that $x\mapsto K(x)$, and hence $x\mapsto {\hat {\M}(x)}$, is homogeneous of degree one.

The Lyapunov exponent $\lambda(\hat\Sigma)$ satisfies the following property, which generalizes the corresponding one for switching systems recalled in
Remark~\ref{rem:pseudo-fenichel}. \\

\begin{lemma}
\label{lem:ES-DI}
System $\hat\Sigma$ is ES if and only if $\lambda(\hat\Sigma)<0$.
\end{lemma}
\begin{proof}
The direct implication is trivial.
On the other hand, by definition of $\lambda(\hat\Sigma)$, if the latter is negative then for every $\delta\in (0,|\lambda(\hat\Sigma)|)$ there exists $T>0$ such that $|\hat x(t)|\leq e^{-\delta t}|\hat x(0)|$ for every solution $\hat x(\cdot)$ of the differential inclusion and $t\geq T$. Furthermore, $|\dot{\hat x}(t)|\leq \hat C|\hat x(t)|$ for some $\hat C>0$, because of the compactness, continuity, and homogeneity of the 
map $\hat{\M}$.
Hence $|\hat x(t)|\leq e^{\hat C t}|\hat x(0)|\leq \hat K e^{-\delta t}|\hat x(0)|$ on $[0,T]$, where $\hat K = \max_{t\in [0,T]}e^{(\hat C+\delta) t}=e^{(\hat C+\delta) T}$. We conclude that $|\hat x(t)|\leq \hat Ke^{-\delta t}|\hat x(0)|$ for every $t>0$.
\end{proof}

The following theorem summarizes the main results obtained in this paper.\\

\begin{theorem}\label{thm:main}
Suppose that Assumption~\ref{UGES-fast} holds. Then  
\begin{equation}\label{eigenvalues}
\lambda(\bar\Sigma) \leq\lambda(\check\Sigma) \leq \liminf_{\e\to 0^+} \lambda(\Sigma_{\e})\leq \limsup_{\e\to 0^+}\lambda(\Sigma_{\e})\leq \lambda(\hat\Sigma).
\end{equation}
Moreover, 
\begin{enumerate}
\item if system~$\check\Sigma$ is EU then system $\Sigma$ is $\e$-EU;
\item If system~$\hat\Sigma$ is ES then system~$\Sigma$ is $\e$-ES.
\end{enumerate}
\end{theorem}


\section{Comparison between the Lyapunov exponents of $\Sigma$ and $\bar\Sigma$}\label{sec:bar}

We prove in this section that $\lambda(\bar\Sigma) \leq \liminf_{\e\to 0^+} \lambda(\Sigma_\e)$ (Proposition~\ref{epsbehavior}). 
Strictly speaking, in view of Proposition~\ref{check} given below, this is not necessary for the proof of Theorem~\ref{thm:main}, but we prefer to provide the proof of this result since it illustrates, in a simplified framework, some of the 
ideas used in the proof of the inequality $\lambda(\check\Sigma) \leq \liminf_{\e\to 0^+} \lambda(\Sigma_\e)$.

The following lemma recalls a rather classical relation between Lyapunov exponents of linear switching systems that will be useful in the sequel. We provide its proof for completeness. \\

\begin{lemma}\label{shift}
Consider a linear switching system~$\Sigma_{\N}$ as in \eqref{general}.
Then, for every $\mu\in\R$, $\lambda(\Sigma_{{\N}+\mu I_n})=\mu +\lambda(\Sigma_{\N})$. 
\end{lemma}

\begin{proof}
Let $\mu\in\R$. Observe that $t\mapsto x(t)$ is a trajectory of~$\Sigma_{\N}$ if and only if $t\mapsto e^{\mu t} x(t)$ is a trajectory of 
$\Sigma_{{\mathcal N}+\mu I_n}$. 
By consequence, 
\begin{align*}
    \lambda(\Sigma_{{\mathcal N}+\mu I_n})&=\limsup_{t\to+\infty}\frac{1}{t}\sup_{N\in \S_{\N+\mu I_n}}\log\|\Phi_{N}(t,0)\|\\
    &=\limsup_{t\to+\infty}\frac{1}{t}\sup_{N\in \S_{\N}}\log\|e^{\mu t}\Phi_{N}(t,0)\|\\
    &=\mu +\lambda(\Sigma_{\mathcal N}),
\end{align*} 
concluding the proof.
\end{proof}

\begin{proposition}\label{epsbehavior}
Suppose that Assumption~\ref{UGES-fast} holds. Then $\lambda(\bar\Sigma)\leq \liminf_{\e\to 0^+} \lambda(\Sigma_\e)$. 
\end{proposition}


\begin{proof}
For $\mu\in\R$ and $\e>0$ consider the switching systems  
\begin{equation}\label{xypdot}
\Sigma_{\e,\mu}:\quad \begin{array}{rll}
\dot x(t)&=&A_{\mu}(t)x(t)+B(t)y(t), \\
\varepsilon \dot y(t)&=&C(t)x(t)+D_{\mu}(t)y(t), 
\end{array}
\end{equation}
and 
\begin{align*}
\bar\Sigma_{\mu}:  \dot {\bar x}(t)&=\left(A_{\mu}(t)-B(t)D^{-1}(t)C(t)\right)\bar x(t), 
\end{align*}
where {$(\begin{smallmatrix}A&B\\C&D\end{smallmatrix})\in \S_\M$,}
\[A_{\mu}(t)=A(t)+\mu I_{n}, \quad \mbox{and} \quad D_{\mu}(t)=D(t)+\e\mu I_{m}.\] 
It follows from Lemma~\ref{shift} that 
\begin{equation*}
    \lambda(\Sigma_{\e,\mu})= \lambda(\Sigma_{\e})+\mu \quad {\rm and } \quad \lambda(\bar \Sigma_{\mu})= \lambda(\bar\Sigma)+\mu.
\end{equation*}
Fix, for now, $\mu\in \R$, $T>0$, and a piecewise-constant switching signal $\s=(\begin{smallmatrix}A&B\\C&D\end{smallmatrix})\in \S_\M$. 
Denote by $t_1<\dots<t_N$ the switching instants of $\s$ within the interval $[0,T]$ and set $t_0=0$, $t_{N+1}=T$. 
Following a classical approach (see e.g. \cite{kokotovic1975riccati}), we introduce the variable
\begin{equation}\label{trans}
z(t)=y(t)+D_{\mu}^{-1}(t)C(t)x(t)+\varepsilon P_{\e}(t) x(t), 
\end{equation}
where $P_{\e}$ is 
constant on each interval $[t_k,t_{k+1})$, $k=0,\dots,N$, and is
chosen in such a way 
that $\Sigma_{\e,\mu}$ is equivalently represented on each interval $[t_k,t_{k+1})$ as
\begin{align} 
\dot x(t)&=\left(M_{\mu}(t)-\e B(t)P_{\e}(t)\right)x(t)+B(t)z(t),\label{xeqdot}\\
\e \dot z(t)&=\left(D_{\mu}(t)+\e Q_{\e}(t)B(t)\right)z(t), \label{yeqdot} 
\end{align} 
where $M_{\mu}=A_{\mu}-BD_{\mu}^{-1}C$, $Q_{\e}=D_{\mu}^{-1}C+\e P_{\e}$, and $\|P_\e(t)\|$ is upper bounded uniformly with respect to $t\in [0,T]$ and $\e$ small enough.
 System~\eqref{xeqdot}-\eqref{yeqdot} is discontinuous at the instants of switching, since the variable $z$ depends on $\sigma$.
Recall that $\Phi_{\e^{-1}{D_\mu}}$ denotes
the flow associated with $\frac{1}{\e}D_{\mu}$. Thanks to Assumption~\ref{UGES-fast},  there exist $c,\alpha>0$ such that for $\e$ small enough and for every $0<s<t$
\begin{equation}\label{phieps}
    \|\Phi_{\e^{-1}{D_\mu}}
(t,s)\|\leq ce^{-\frac{\alpha}{\e} (t-s)}. 
\end{equation}
Observe that there exists a positive constant $K$ 
independent of $\e$ (but possibly depending on $T,\M,\mu$) 
such that for every $(x_0,y_0)\in\R^n\times\R^m$ and  every $\e>0$ small enough
one has 
\begin{equation}\label{eq:wiK}
|(x(t),z(t))|\leq K|(x_0,y_0)|,\qquad t\in [0,T],
\end{equation}
where $(x(\cdot),y(\cdot))$ is the trajectory of $\Sigma_{\e,\mu}$ 
associated with $\s$ and the initial condition $(x_0,y_0)$ and $z(\cdot)$ is given by \eqref{trans}.
 By a slight abuse of notation, in what follows we still use $K$ to denote possibly larger constants 
independent of $\e$. 
For $t\in [t_{k},t_{k+1})$, by applying the variation of constant formula to \eqref{yeqdot} and using~\eqref{trans} for $t=t_k$,  we have 
\begin{align*}
    z(t)={}&
    \Phi_{\e^{-1}{D_\mu}}(t,t_{k})z(t_k)
    +\int_{t_{k}}^{t}{
    \Phi_{\e^{-1}{D_\mu}}(t,s)Q_{\e}(s)B(s)z(s)}ds.
\end{align*}
As an immediate consequence of~\eqref{phieps} and \eqref{eq:wiK}, one gets that 
\begin{equation}\label{estimate1}
|z(t)|\leq (e^{-\frac{\alpha}{\e}(t-t_k)}+\e) K |(x_0,y_0)|, \,\forall t\in [t_k,t_{k+1}).   
\end{equation}
By applying the variation of constant formula to~\eqref{xeqdot}, one deduces that, 
for every $t\in[0,T]$, 
\begin{align*}
x(t)-\bar x(t)
&=\int_{0}^{t}\Phi_{\bar M_{\mu}}(t,s)B(s)
\left(D^{-1}(s)-D_{\mu}^{-1}(s)\right)C(s)x(s)ds
-\e\int_{0}^{t}\Phi_{\bar M_\mu}(t,s)B(s) P_{\e}(s)x(s)ds\\
&+\int_{0}^{t}\Phi_{\bar M_\mu}(t,s)B(s)z(s)ds, 
\end{align*}
where 
$\bar M_{\mu}:=A_{\mu}-BD^{-1}C$.
By the definition of $D_{\mu}$ we have that
\begin{equation}
    \label{Dmu}
    \|D^{-1}(t)-D_{\mu}^{-1}(t)\|\le \e K, \qquad \forall\, t\in [0,T].
\end{equation}
Hence, using estimates ~\eqref{eq:wiK} and \eqref{estimate1}, one deduces that 
\begin{equation}\label{estimate2}
|x(t)-\bar x(t)
|\leq \e K|(x_0,y_0)|.
\end{equation}

From~\eqref{trans} together with  \eqref{estimate1}, \eqref{Dmu}, and \eqref{estimate2} we obtain that 
\begin{align}
|y(T)+D^{-1}_{N}C_{N}\bar x(T)|\nonumber
&\leq
|z(T)|+K|x(T)-\bar x(T)|+
\e K |(x_0,y_0)|\leq 
(e^{-\frac{\alpha}{\e}(T-t_N)}+\e) K |(x_0,y_0)|\nonumber\\
&\le \e  K |(x_0,y_0)|\label{estimate4}
\end{align}
for $\e$ sufficiently small, where $ (\begin{smallmatrix} A_N&B_N\\C_N&D_N\end{smallmatrix})$ is the value of $\s$ on $[t_N,T]$.


If $\Phi^\e_{\s,\mu}$ denotes the flow of~\eqref{xypdot} associated with $\s$,  one deduces from~\eqref{estimate2} and~\eqref{estimate4} that 
\begin{equation*}
\left\|\Phi^\e_{\s,\mu}(T,0)-\left(\begin{array}{lll}\Phi_{\bar M_\mu} (T,0) & 0\vspace{3pt}\\  -D_N^{-1}C_{N}\Phi_{\bar M_\mu} (T,0)&0\end{array}\right)\right\|\leq \e K.
\end{equation*}

Let $\mu$ be any constant such that
$\lambda(\bar\Sigma)+\mu=\lambda(\bar\Sigma_\mu)>0$, and choose the time 
$T>0$ and the piecewise-constant switching signal $\s$ 
so that the spectral radius of $\Phi_{\bar M_\mu}(T,0)$ is larger than one (this is possible thanks to \eqref{eq:JSR}). 
By continuity of the spectral radius, 
it follows that 
$\rho(\Phi^\e_{\s,\mu}(T,0))$ is also larger than one for $\e$ small enough.
Then one deduces from \eqref{eq:gelfand} that $\lambda(\Sigma_\e)+\mu=\lambda(\Sigma_{\e,\mu})> 0$. 
The conclusion follows by arbitrariness of $\mu>-\lambda(\bar \Sigma)$. 
\end{proof}

%
%
%
%
%
\section{Comparison between the Lyapunov exponents of $\Sigma$ and $\check\Sigma$}\label{czech}

We have the following preliminary results.

\begin{lemma}\label{paolo-lemma}
Under Assumption~\ref{UGES-fast} the set $\mathcal{\check M}$ is bounded.
\end{lemma}
\begin{proof}
By Assumption~\ref{UGES-fast} we have that 
\begin{equation}\label{eq-flow-bound}
\|\Phi_D(t_2,t_1)\|\leq ce^{-\alpha (t_2-t_1)}
\end{equation} 
for every $t_2\geq t_1$, for some constants $\alpha>0$ and $c\geq 1$ independent of the switching law $D\in \S_{\M_D}$. As a consequence of \eqref{eq-flow-bound} and the boundedness of $\M$, we have
$\|\Lambda_0(T,\sigma)\|\leq C_1 \min\{1,T\}$ and $\|\Lambda_2(T,\sigma)\|\leq C_1 \min\{1,T\}$ for some $C_1>0$, and $\|\Lambda_1(T,\sigma)\|\leq C_2 T$ for some $C_2>0$. Moreover 
since $\rho(\Phi_D(T,0))<1$ for every $T>0$ and $D\in \S_{\M_D}$ according to \eqref{eq:gelfand}, $\left(I_m-\Phi_D(T,0)\right)^{-1}$ is well defined and expandable as an absolutely convergent power series in $\Phi_D(T,0)$. Letting $\bar T = \frac{\log(2 c)}{\alpha}$ so that $\|\Phi_D(\tau,0)\|\leq\frac12$ for every $\tau\geq \bar T$, we can write
\begin{align*}
(I_m-&\Phi_D(T,0))^{-1} =\sum_{k\geq 0} \Phi_D(T,0)^k \\
& =\sum_{h=0}^{\floor{\frac{\bar T}{T}}} \Phi_D(T,0)^h + \Phi_D(T,0)^{\floor{\frac{\bar T}{T}}+1} \left(\sum_{h=0}^{\floor{\frac{\bar T}{T}}} \Phi_D(T,0)^h\right) \\
&+ \Phi_D(T,0)^{2({\floor{\frac{\bar T}{T}}+1})} \left(\sum_{h=0}^{\floor{\frac{\bar T}{T}}} \Phi_D(T,0)^h\right)+\dots\\
&= \left(\sum_{k\geq 0} \Phi_D(T,0)^{(\floor{\frac{\bar T}{T}}+1)k}\right)\left(\sum_{h=0}^{\floor{\frac{\bar T}{T}}} \Phi_D(T,0)^h\right).
\end{align*}
As $\Phi_D(T,0)^h$ corresponds to the flow of $\Sigma_D$ 
for a $T$-periodic signal at time $hT$ we have $\|\Phi_D(T,0)^h\|\leq c$ by \eqref{eq-flow-bound}. Hence
\[
\left\|\sum_{h=0}^{\floor{\frac{\bar T}{T}}} \Phi_D(T,0)^h\right\|\leq \left(1+\floor{\frac{\bar T}{T}}\right)c\leq c + \frac{c\bar T}{T}.
\]
Moreover, as $(\floor{\frac{\bar T}{T}}+1)T\geq \bar T$,
\begin{align*}
\left\|\sum_{k\geq 0} \Phi_D(T,0)^{(\floor{\frac{\bar T}{T}}+1)k}\right\|&\leq \sum_{k\geq 0}\|\Phi_D(T,0)^{(\floor{\frac{\bar T}{T}}+1)}\|^k\leq \sum_{k\geq 0}\frac1{2^k}=2.
\end{align*}
Summing up
\begin{align*}
\|\Lambda(T,\s)\|&\leq C_2+\frac{C_1^2\min\{1,T\}^2 (2c+2c\bar{T}/T)}{T}\\
&=C_2+2cC_1^2\min\left\{\frac1T+\frac{\bar T}{T^2},T+\bar T\right\}\\
&\leq C_2+2cC_1^2(1+\bar T),
\end{align*}
concluding the proof of the lemma.
\end{proof}

\begin{lemma}\label{lem:smile}
Let $T>0$, $\sigma = (\begin{smallmatrix}
A&B\\
C&D
\end{smallmatrix})\in \S_\M$, and $\mu\in\R$. 
For every $\e>0$ denote by $\mathscr{M}(\e)$ the flow at time $\e T$ of $\Sigma_{\e,\mu}$ defined in \eqref{xypdot} and  corresponding to the signal $\sigma(\cdot/\e)$.
Then there exists $P(\e)$ of the form 
\begin{equation}
\label{eq:Peps}    
P(\e)=\begin{pmatrix}I_n & 0\\
Q(\e) & I_m \end{pmatrix}
\end{equation}
such that
\begin{align}
P(\e)^{-1}\mathscr{M}(\e)P(\e)=
&\begin{pmatrix} I_n+\e T (\Lambda(T,\sigma)+\mu I_n)+O(\e^2)&O(\e)\\
0&\Phi_D(T,0)+O(\e) \end{pmatrix},\label{eq: ptrans}
\end{align}
where 
\begin{equation}\label{Qe}
Q(\e)=\Big(I_m-\Phi_D(T,0)\Big)^{-1}\Lambda_0(T,\sigma)+O(\e)
\end{equation}
and the functions $O(\e^k)$ are such that $\|O(\e^k)\|\le C\e^k$ for $\e$ small, for some positive constant $C$ independent of $\e$ and $\s$. 
\end{lemma}
\begin{proof}
We first prove that $\e\mapsto \mathscr{M}(\e)$ admits a first order expansion 
\begin{equation}\label{DL}
\mathscr{M}(\e)=\mathscr{M}_{0}+\e \mathscr{M}_{1}+O(\e^2),
\end{equation}
for some matrices $\mathscr{M}_{0},\mathscr{M}_{1}$ to be computed.
To see that, we first apply the time rescaling $\tau=t/\e$ and 
we have that $\mathscr{M}(\e)$ is equal to $\Phi_{N_{0}+\e N_{1}}(T,0)$,
where the signals $N_0$ and $N_1$ are defined as
\begin{equation*}
N_{0}(\tau)=
\begin{pmatrix}
0 & 0\\
C(\tau) & D(\tau)
\end{pmatrix},
\, 
N_{1}(\tau)=
\begin{pmatrix}
 A(\tau)+\mu I_n &  B(\tau)\\
0 & \mu I_m
\end{pmatrix}.
\end{equation*}
By the variation of constant formula, one has
\begin{align*}
\Phi_{{N}_{0}+\e N_1}(T,0)= 
 \Phi_{{N}_{0}}(T,0) +\e \int_0^T \Phi_{{N}_{0}}(T,\tau)N_1(\tau)\Phi_{{N}_{0}+\e N_1}(\tau,0)\;d\tau.
\end{align*}
One deduces that 
\eqref{DL} holds true with
\begin{align*}
\mathscr{M}_{0}=\Phi_{N_{0}}(T,0) \quad \mbox{and} \quad \mathscr{M}_{1}=\int_0^T\Phi_{{N}_{0}}(T,\tau)N_{1}(\tau)
\Phi_{{N}_{0}}(\tau,0)\;d\tau.
\end{align*}
It is easy to get that 
\begin{align}
\mathscr{M}_{0}=\begin{pmatrix}I_n&0\\ \Lambda_0(T,\sigma)&
\Phi_D(T,0)\end{pmatrix},\quad \mbox{and} \quad 
 \mathscr{M}_{1}=\begin{pmatrix}\Lambda_1(T,\sigma)+T\mu I_n&\Lambda_2(T,\sigma)\\ \Lambda_3(T,\sigma)&
\Lambda_4(T,\sigma)\end{pmatrix},\label{eq:M1}
    \end{align}
where
\begin{align*}\Lambda_3(T,\sigma)&=\Lambda_0(T,\sigma)\left(\Lambda_1(T,\sigma)+T\mu\right)
-\int_0^T \Phi_D(T,\tau) \Lambda_0(\tau,\s)(A(\tau)+B(\tau)\Lambda_0(\tau,\s))d\tau
\end{align*}
and
\begin{align*}
\Lambda_4(T,\sigma)&=\mu T\Phi_D(T,0)+\Lambda_0(T,\sigma)\Lambda_2(T,\sigma)
-\int_0^T \Phi_D(T,\tau) \Lambda_0(\tau,\s)B(\tau)\Phi_D(\tau,0)d\tau.
\end{align*}
According to~\eqref{DL} and~\eqref{eq:M1}, it 
follows
that~\eqref{eq: ptrans} holds true with 
$P(\e)$ and $Q(\e)$ as in \eqref{eq:Peps} and \eqref{Qe}.
This concludes the proof of the lemma. 
%
\end{proof}

We can now prove the main result of this section.
 
\begin{proposition}
\label{check}
Suppose that Assumption~\ref{UGES-fast} holds. Then $\lambda(\bar\Sigma)\leq \lambda(\check\Sigma)\leq \liminf_{\e\to 0^+} \lambda(\Sigma_\e)$. Moreover, if~$\check\Sigma$ is EU then~$\Sigma$ is $\e$-EU. 
\end{proposition}

\begin{proof}
We first prove the  inequality $\lambda(\bar\Sigma)\leq \lambda(\check\Sigma)$ by showing  that  $
\bar{\mathcal{M}}\subset \mathcal{\check M}$.
To see that, let us check that $A-BD^{-1}C$ belongs to 
$\mathcal{\check M}$ for every  $M=(\begin{smallmatrix} A&B\\C&D\end{smallmatrix})$ in $\mathcal{M}$. Indeed, letting 
$T>0$ and  $\bar{\sigma}\in\S_\M$ constantly equal to $M$ on $[0,T]$, it holds
\begin{equation*}
\Lambda(T,\bar{\sigma})=A-BD^{-1}C.
\end{equation*}

%

Assume that $\mu$ is chosen so that 
\[
\lambda(\check\Sigma)+\mu>0.
\]
Then, according to \eqref{eq:JSR}, there exist $\ell\in\mathbb{N}$, $\ell$ matrices $\Lambda(T_1,\sigma_1),\dots,\Lambda(T_\ell,\sigma_\ell)\in \check{\mathcal{M}}$,
and $\ell$ positive times $t_1,\dots,t_\ell$ so that 
\begin{equation}\label{eq:bigger1}
\rho\Big(e^{t_\ell(\Lambda(T_\ell,\sigma_\ell)+\mu I_n)} \dots e^{t_1(\Lambda(T_1,\sigma_1)+\mu I_n)}\Big)>1.
\end{equation}

Let $\e$ be sufficiently small so that 
$\e T_k<t_k$
for every $k=1,\dots,\ell$,   and denote by 
 $N_k=\floor{\frac{t_k}{\e T_k}}$  the number of intervals of length $\e T_k$ contained in $[0,t_k]$. 

Then, for every $k=1,\dots, \ell$, consider the flow 
$\mathscr{M}_k(\e)$
of system $\Sigma_{\e,\mu}$ (cf.~\eqref{xypdot}) corresponding to the signal $\sigma_k(\cdot/\e)$, evaluated at time  
$\e{T_k}$.
Thanks to Lemma~\ref{lem:smile}, there exists $P_k(\e)$ given by 
$P_k(\e)=\begin{pmatrix}I_n & 0\\
Q_k(\e) & I_m \end{pmatrix}$ 
such that
\begin{align*}
\mathcal{T}_k(\e)=P_k(\e)^{-1}\mathscr{M}_k(\e)P_k(\e)
\end{align*}
satisfies 
\begin{align*}
\mathcal{T}_k(\e)=
\begin{pmatrix} I_n+\e T_k (\Lambda(T_k,\sigma_k)+\mu I_n)+O(\e^2)&O(\e)\\
0&\Phi_D(T_k,0)+O(\e) \end{pmatrix}.
\end{align*}
We repeat $N_k$ times $\sigma_k(\cdot/\e)$ to get a signal on $[0,\e T_k N_k]$ and the corresponding flow of $\Sigma_{\e,\mu}$ at time $\e T_k N_k$ is given by 
\begin{equation}\label{eq:Mk}
\mathscr{M}_k(\e)^{N_k}=P_k(\e)\mathcal{T}_k(\e)^{N_k}P_k(\e)^{-1}.
\end{equation}
We claim that
\begin{equation}\label{eq:puissanceN1}
\mathcal{T}_k(\e)^{N_k}=
\begin{pmatrix}e^{ t_k (\Lambda(T_k,\sigma_k)+\mu I_n)}+O(\e)&O(\e)\\ 0&O(\e)
\end{pmatrix}.
\end{equation}
This follows from the general formula
\[\begin{pmatrix}A_{11}&A_{12}\\0&A_{22}\end{pmatrix}^N=\begin{pmatrix}A_{11}^N&W\\0&A_{22}^N\end{pmatrix},\,
W=\sum_{j=0}^{N-1} A_{11}^jA_{12}A_{22}^{N-j-1},\]
applied to
 $N=N_k$, 
 \[A_{11}=I+B_{11}/N_k+O(\e^2)=e^{\frac{B_{11}}{N_k}}+O(\e^2)\] with $B_{11}=t_k (\Lambda(T_k,\sigma_k)+\mu I_n)$, and
 \[A_{22}=\Phi_D(T_k,0)+O(\e).\] Then, for $\e$ small enough one gets that
\begin{eqnarray}
\|A_{11}^j\|&\le& \Big(1+\frac{2\|B_{11}\|}j\Big)^j\le e^{2\|B_{11}\|},\quad 1\le j\le N_k,\label{eq:A11-1}\\
A_{11}^{N_k}&=&e^{B_{11}}+O(\e).\nonumber 
\end{eqnarray}
Moreover, since $\rho(A_{22})<1$ for $\e$ small, it follows that
\begin{equation}\label{eq:A22}
\|A_{22}^j\|\leq K\lambda^j,\quad 1\leq j\leq N_k,
\end{equation}
for some $K>0$ and $\lambda\in (0,1)$ independent of $\e$ small enough and $k\in\{1,\dots\ell\}$. In particular, \[\|A_{22}^{N_k}\|\leq K\lambda^{\frac{t_k}{\e T_k}-1}=O(\e).\]
Using now \eqref{eq:A11-1}, \eqref{eq:A22}, and recalling that $K$ can be taken so that $\|A_{12}\|\leq K\e$, one deduces that
\begin{align*}
\|W\|\le \sum_{j=0}^{N_k-1}\|A_{11}^j\|\|A_{12}\| \|A_{22}^{N_k-j-1}\|
\le K^2e^{2\|B_{11}\|}\e \sum_{j=0}^{N_k-1} \lambda^{N_k-1-j}=C\e,
\end{align*}
for some $C>0$ independent of $\e$ small enough and $k\in\{1,\dots\ell\}$. This concludes the proof of \eqref{eq:puissanceN1}.

We next use \eqref{Qe} and \eqref{eq:puissanceN1} in \eqref{eq:Mk} to deduce that 
\begin{equation}\label{eq:good-Mk}
\mathscr{M}_k(\e)^{N_k}=\begin{pmatrix}e^{ t_k (\Lambda(T_k,\sigma_k)+\mu I_n)}&0\\
r_k & 0\end{pmatrix}+O(\e),
\end{equation}
where \[r_k=\Big(I_m-\Phi_{D_{\sigma_k}}(T_k,0)\Big)^{-1}\Lambda_0(T_k,\sigma_k)e^{t_k(\Lambda(T_k,\sigma_k)+\mu I_n)}.\] 

Set $t_{\e}=\e\sum_{k=1}^\ell N_kT_k$
 and notice that $t_{\e}$ tends to $t_1+\dots+t_\ell$ as $\e$ tends to zero. 
 We concatenate the $N_k$ times repetitions of $\sigma_k(\cdot/\e)$ for $k=1,\dots,\ell$ to get a signal on $[0,t_{\e}]$ and the corresponding flow of $\Sigma_{\e,\mu}$ at time $t_{\e}$ is given by the matrix product
\[\Upsilon_{\e}=\mathscr{M}_\ell(\e)^{N_\ell}\mathscr{M}_{\ell-1}(\e)^{N_{\ell-1}}\cdots \mathscr{M}_{1}(\e)^{N_{1}}.\]
Using \eqref{eq:good-Mk}, one deduces that
\[
\Upsilon_{\e}=\begin{pmatrix}
e^{t_\ell(\Lambda(T_\ell,\sigma_\ell)+\mu I_n)} \cdots e^{t_1(\Lambda(T_1,\sigma_1)+\mu I_n)}&0\\
r & 0\end{pmatrix}+O(\e),
\]
where the matrix $r$ does not depend on $\e$. 

Using \eqref{eq:bigger1}, one gets that  $\rho(\Upsilon_{\e})>1$ for $\e$ small enough, yielding that 
\[\lambda(\Sigma_\e)+\mu=\lambda(\Sigma_{\e,\mu})>0.
\]
 In particular $\liminf_{\e\to 0^+}\lambda(\Sigma_\e)+\mu\ge 0$. By letting $\mu$ tend to $-\lambda(\check\Sigma)$ from above, one concludes that $\lambda(\check\Sigma) \leq \liminf_{\e\to 0^+} \lambda(\Sigma_\e)$ as desired. 

We are left to show that if~$\check\Sigma$ is EU then~$\Sigma$ is $\e$-EU. 
For this purpose, we take $\mu=0$ in the previous calculations and we observe that the flow $\Upsilon_{\e,t}$ of $\Sigma_\e$ at time $t\in [0,t_{\e}]$ corresponding to the signal defined above satisfies
\begin{align}
\Upsilon_{\e,t}=
\begin{pmatrix}
e^{(t-\sum_{h=1}^{k-1}t_h)\Lambda(T_k,\sigma_k)} \prod_{\ell=1}^{k-1}e^{t_{\ell}\Lambda(T_{\ell},\sigma_{\ell})}+O(\e) & O(\e)\\
r(\e,t) & q(\e,t)\end{pmatrix},\label{ups!}
\end{align}
whenever  $t\in [\sum_{h=1}^{k-1}t_h,\sum_{h=1}^{k}t_h]$, for some matrix functions $r,q$, 
where the terms $O(\e)$ are uniform with respect to $t\in [0,t_{\e}]$. 
The matrix $\Upsilon_{\e,t_\e}$ converges, as $\e$ goes to zero, to 
\[\bar\Upsilon = \begin{pmatrix} e^{t_\ell \Lambda(T_\ell,\sigma_\ell)} \cdots e^{t_1 \Lambda(T_1,\sigma_1)} & 0\\ \bar r & 0\end{pmatrix},\] 
for some matrix $\bar r$.
For $\e$ small enough we construct a trajectory $z_\e(t) = (x_\e(t),y_\e(t))$ of $\Sigma_\e$ satisfying $|z_\e(t)| \geq \hat C e^{\hat\lambda t}|z_\e(0)|>0$ for every $t\geq 0$, with $\hat\lambda \in (0,\lambda(\check\Sigma))$ 
and $\hat C>0$ independent of $\e$. Let {$\pi$} be the sum of the projectors on the generalized eigenspaces associated with the eigenvalues of  $\bar\Upsilon$ of modulus $\rho(e^{t_\ell \Lambda(T_\ell,\sigma_\ell)} \cdots e^{t_1 \Lambda(T_1,\sigma_1)})>1$.
Since $\Upsilon_{\e,t_\e}$ converges to $\bar\Upsilon$ as $\e$ goes to zero, by classical results (see~\cite[Theorem 5.1, Chapter II]{kato}) there exists $\pi_\e$, a sum of projectors on generalized eigenspaces of $\Upsilon_{\e,t_\e}$, satisfying $\lim_{\e\to 0}\pi_\e = \pi$, and the corresponding eigenvalues also converge. Let $v_\e$ be a possibly complex eigenvector of the restriction of $\Upsilon_{\e,t_\e}$ to the image of $\pi_\e$, associated with an eigenvalue $\alpha_\e$. If $\alpha_\e$ is real then $v_\e$ can be taken real as well, otherwise we assume without loss of generality that $v_\e$ 
satisfies $|\mathrm{Re}(v_\e)| = \min_{\theta\in\R}|\mathrm{Re}(e^{i\theta}v_\e)|$.
In particular $|\mathrm{Re}(\beta v_\e)|\geq |\beta| |\mathrm{Re}(v_\e)|$ for every $\beta\in\mathbb{C}$. Note that, for every positive integer $k$, one has that $(\Upsilon_{\e,t_\e})^k v_\e = \alpha_\e^k v_\e$ and $(\Upsilon_{\e,t_\e})^k\bar{v}_\e =\bar{\alpha}_\e^k \bar{v}_\e$ which implies $(\Upsilon_{\e,t_\e})^k \mathrm{Re}(v_\e) = \mathrm{Re}(\alpha_\e^k v_\e)$. 

Consider the trajectory $z_\e(t) = (x_\e(t),y_\e(t))$ of $\Sigma_\e$ obtained applying the flow $\Upsilon_{\e,t}$ to the initial condition $z_\e(0)=\mathrm{Re}(v_\e)$ and repeating periodically after time $t_\e$. Letting $\Pi_x$ be the projection of a vector of $\R^{n+m}$ onto its first $n$ components, it is easy to see that $|\Pi_x v|\geq C|v|$ for every $v$ in the image of $\bar\Upsilon$, where 
\[C=(1+\|\bar r\|^2\|(e^{t_\ell \Lambda(T_\ell,\sigma_\ell)} \cdots e^{t_1 \Lambda(T_1,\sigma_1)})^{-1}\|^2)^{-1/2}.\]
Hence, for every nonnegative integer $h$ and $\e$ small enough, 
\begin{align}
|x_\e(ht_\e)|&= |\Pi_x z_\e(ht_\e)|\nonumber\\
&= |\Pi_x \pi_\e z_\e(ht_\e)|\nonumber\\
&\geq |\Pi_x \pi z_\e(ht_\e)| -\|\Pi_x\|\|\pi-{\pi}_\e\|\,|z_\e(ht_\e)|\nonumber\\
&\geq |\Pi_x \pi z_\e(ht_\e)| -\|\pi-\pi_\e\|\,|z_\e(ht_\e)|\nonumber\\
&\geq C|\pi z_\e(ht_\e)|-\|\pi-{\pi}_\e\|\,|z_\e(ht_\e)|\nonumber\\
&\geq C|z_\e(ht_\e)|-(1+C)\|{\pi}-{\pi}_\e\|\,|z_\e(ht_\e)|\nonumber\\
&\geq \frac{C}{2}|z_\e(ht_\e)|.\label{eq-periodic-1}
\end{align}
As $z_\e(ht_\e) = \mathrm{Re}(\alpha_\e^h v_\e)$ we also have
\begin{equation}
\label{eq-periodic-2}
|x_\e(ht_\e)|\geq \frac{C}{2}|\alpha_\e|^h |z_\e(0)|.
\end{equation}
By~\eqref{ups!}, \eqref{eq-periodic-1}, and setting $\kappa = \max_{\Lambda\in \mathcal{\check M}} \|\Lambda\|$ it follows that 
\begin{align}\label{eq-periodic-3}
|x_\e(ht_\e+\tau)|&\geq |x_\e(ht_\e)|e^{-\kappa \tau}-|z_\e(ht_\e)|O(\e)\nonumber\\
&\geq  |x_\e(ht_\e)|e^{-\kappa \tau}-\frac{2}{C}|x_\e(ht_\e)|O(\e)\nonumber\\
&\geq \frac{1}{2}|x_\e(ht_\e)|e^{-\kappa \tau}
\end{align}
for every $\tau\in [0,t_\e]$. Take $\e$ small enough in such a way that $1<\underline\alpha\leq |\alpha_\e|\leq \bar\alpha$ for some positive constants $\underline\alpha,\bar\alpha$, and $ t_\e\leq \bar t$, where $\bar t= 2 \sum_{h=1}^\ell t_h$. By~\eqref{eq-periodic-2} and \eqref{eq-periodic-3} we thus get for every $t\geq 0$
\begin{align*}
|z_\e(t)|\geq |x_\e(t)|&\geq \frac{C}{4}e^{-\kappa t_\e}|\alpha_\e|^{\floor{\frac{t}{t_\e}}}|z_\e(0)|
\geq  \frac{Ce^{-\kappa t_\e}}{4|\alpha_\e|}e^{\frac{\log|\alpha_\e|}{t_\e}t}|z_\e(0)|\geq \hat Ce^{\hat\lambda t}|z_\e(0)|,
\end{align*}
where $\hat C=\frac{Ce^{-\kappa\bar t }}{4\bar\alpha}$ and 
$\hat\lambda = \frac{\log\underline\alpha}{\bar t}$, which shows that $\Sigma$ is $\e$-EU.
\end{proof}

\begin{example}[$\check \Sigma$ gives sharper bounds than $\bar \Sigma$]
Consider system $\Sigma$ with 
\[
\M=\left\{M_1=\begin{pmatrix}
    -1& 1 \\
     0& -0.1
\end{pmatrix}, M_2=\begin{pmatrix}
    -3& 0 \\
     {2}& -0.1
\end{pmatrix}\right\}.
\]
The stability of singularly perturbed planar switching systems is completely characterized in~\cite[Theorem 2]{Hachemi2011} through some necessary and sufficient conditions. Based on this characterization (cf., in particular, Item (SP5) in \cite[Theorem 2]{Hachemi2011})  the  condition 
\begin{align}\label{eq:SP5}
    \Gamma(M_1,M_2)&:=\frac{1}{2}\left(\mathrm{tr}(M_1)\mathrm{tr}(M_2)-\mathrm{tr}(M_1M_2)\right)
   <-\sqrt{\det(M_1)\det(M_2)}
\end{align}
implies that $\Sigma_\e$ is EU for all $\e>0$. Condition~\eqref{eq:SP5} is satisfied in the case of this example with $\Gamma(M_1,M_2)=-0.8$ and $\det(M_1M_2)=0.03$. 
Look now at systems $\bar\Sigma$ and $\check\Sigma$.
We have $\bar\M=\{-1, -3\}$ and then the associated system $\bar\Sigma$ is ES. 
Concerning system $\check\Sigma$, let us consider the switching signal 
\[\sigma(t)=\alpha(t)M_1+(1-\alpha(t))M_2
\]
associated with the 2-periodic function 
\begin{equation*}
 \alpha(t)=
 \begin{cases}
    1 & t\in [0,1],  \\
     0 & t\in [1,2],  
 \end{cases}
\end{equation*}
and take $T=2$. For this choice of $\sigma$ and $T$ one can easily verify that \[\Lambda(T,\s)=-2+100(1-e^{-0.2})^{-1}(1-e^{-0.1})^2>0.\] Then $\check\Sigma$ is EU, as illustrated in Figure~\ref{fig1}.

\begin{figure}[!ht]
    \centering
    \includegraphics[scale=0.5]{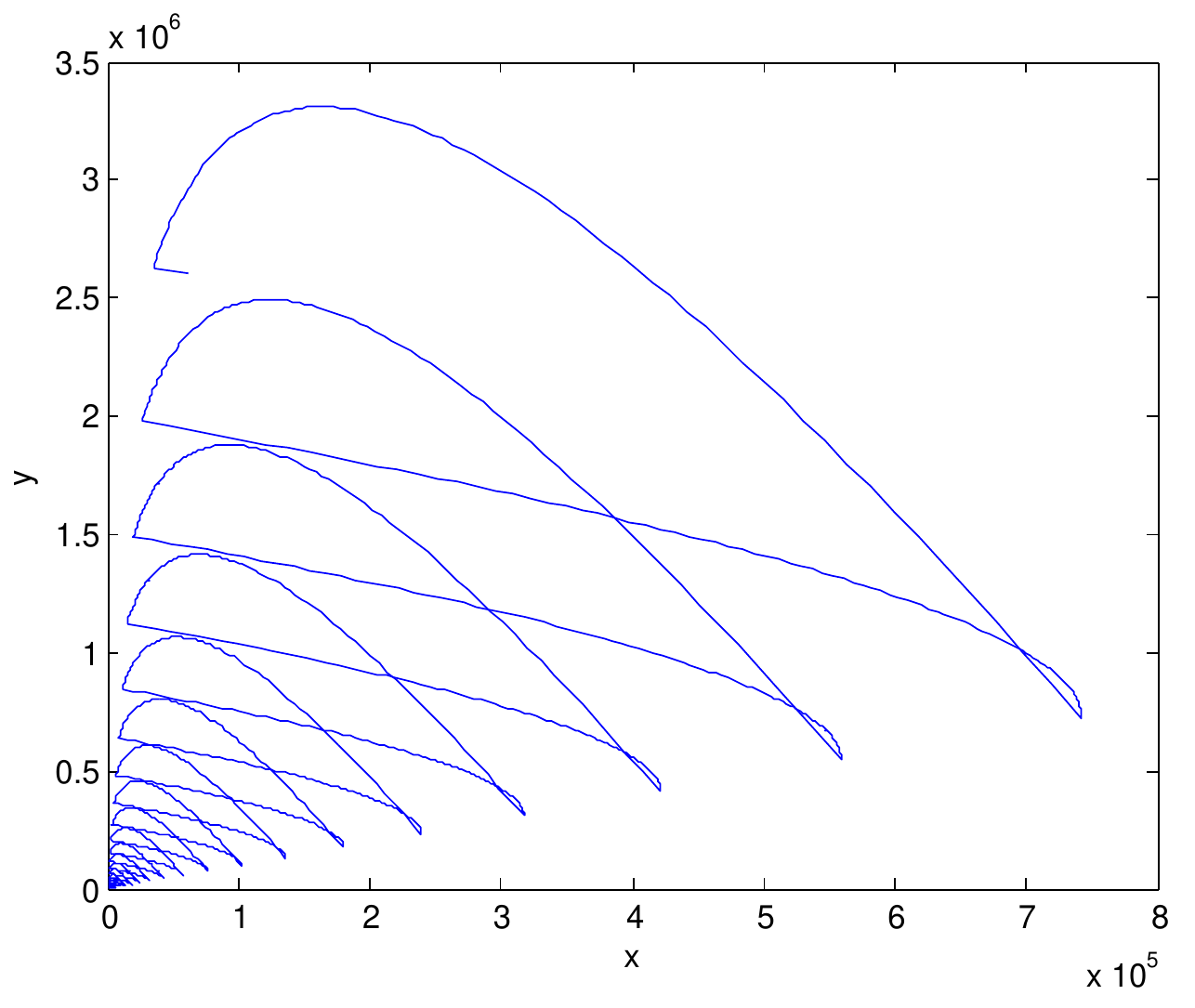}
    \caption{Fast and slow variables evolution of system $\Sigma_\e$ with signal $\sigma$ and $\e=0.1$, starting from $(x_0,y_0)=(1,1)$. }
    \label{fig1}
\end{figure}

\end{example}

\section{Comparison between the Lyapunov exponents of $\Sigma$ and $\hat \Sigma$}\label{sec:hat}

\subsection{Definition and structural properties of the differential inclusion $\hat\Sigma$}
The following lemma studies the $\omega$-limit set of the dynamics $\Sigma_x$ given by~\eqref{fast}.\\
\begin{lemma}\label{independance}
Let $x\in\R^n $. Then, for every $y_0\in\R^m$ and $\s\in\S_{\M}$, the $\omega$-limit set $\omega_{\sigma}^{x}$ of the trajectory of $\Sigma_x$ associated with $\sigma$ and starting at $y_0$ does not depend on the initial condition $y_0$. 
\end{lemma}

\begin{proof}
Let $y_0, y_1\in \R^{m}$ and let $y_0(\cdot),y_1(\cdot)$ be the trajectories of $\Sigma_x$ associated with $\s$ and starting from $y_0$ and $y_1$, respectively. Setting $z(\cdot)=y_0(\cdot)-y_1(\cdot)$, one has $\dot z(t)=D(t)z(t)$, and thanks to Assumption~\ref{UGES-fast}, there exist $c, \delta>0$ such that 
\begin{equation}\label{eq: invariance 1}
|y_0(t)-y_1(t)|\leq ce^{-\delta t}|y_0-y_1|, \quad \forall\, t\geq 0. 
\end{equation}
By definition of an $\omega$-limit set, one deduces at once that $\omega_\sigma^x$ does not depend on $y_0$. 
 \end{proof}

Let us introduce the set valued-map $K:\R^n\leadsto\R^m$ defined by $K(x)=\overline{\bigcup_{\s}\omega_{\s}^{x}}$. We have the following proposition.\\

\begin{proposition}\label{forward invariance}
For each $x\in\R^n $ the set $K(x)$ is compact and forward invariant for the dynamics of $\Sigma_x$, and there exist $c, \delta>0$ such that 
\begin{equation}\label{eq:invariance}
\dist(y_0(t),K(x))\leq ce^{-\delta t}\dist(y_0,K(x)), \quad \forall t\ge 0,\;y_0\in \R^m,
\end{equation} 
for all $\s\in \S_{\M}$, where $y_0(\cdot)$ is the trajectory of~$\Sigma_x$ associated with $\s$ and starting from $y_0$.
\end{proposition}


\begin{proof}
Let $c$ and $\delta$ be as in \eqref{eq: invariance 1}.  
The set  $K(x)$ is closed by definition and its boundedness follows from the fact that 
for every $\sigma\in \S_\M$ the corresponding solution $y(\cdot)$ of $\Sigma_x$ with initial condition $y(0)=0$ satisfies 
\[|y(t)|=\left|\int_0^t \Phi_D(t,s)C(s)x\;ds\right|
\le \frac{c|x|\max_{M\in\M}\|M\|}\delta,\]
for $t\geq0$. Hence $K(x)$ is compact.

Let now $\bar y\in K(x)$, i.e., $\bar y=\lim_{k\to +\infty}y^k$, where $y^k\in \omega_{\s_k}^{x}$, $\s_k\in \S_\M$. For $\tilde t> 0$ and $\tilde\s$, 
denote by $\Phi_{x,\tilde\s}$ the flow of $\Sigma_x$ associated with the signal $\tilde\s$ and 
let us prove that $\tilde y:=\Phi_{x,\tilde\s}(\tilde t,0)\bar y$ is in $K(x)$. For $k\geq 1$, let $r_k=|\tilde y-\Phi_{x,\tilde\s}(\tilde t, 0)y^k)|$ and notice that $\lim_{k\to +\infty} r_k=0$. Moreover, for $k\geq 1$, there exists $\rho_k>0$ such that 
$\Phi_{x,\tilde\s}(\tilde t,0)B_{\rho_k}(y^k)\subset 
B_{2 r_k}(\tilde y)$.
We next define recursively the sequence 
$(z_{k})_{k\geq 0}$ with $z_0\in \R^m$ by setting
\begin{equation*}
    z_{k+1}=
    \Phi_{x,\tilde\s}(\tilde t,0)\Phi_{x,\s_{k+1}}(t_k,0)z_k,\qquad k\ge 0,
\end{equation*}
where the sequence of times $t_k$ 
is chosen so that $\Phi_{x,\s_{k+1}}(t_k,0)z_k
$ is in $B_{\rho_{k+1}}(y^{k+1})$. This is possible since $y^{k+1}\in  \omega_{\s_{k+1}}^{x}$ for $k\geq 0$. 
By construction, $z_k$ is in $B_{2 r_{k}}(\tilde y)$ for every $k\ge 1$. Moreover, $z_k=
\Phi_{x,\bar \s}(\tau_k,0)z_0$
where $\tau_k\to \infty$ and $\bar \sigma$ is constructed by repeatedly concatenating $\sigma_k|_{[0,t_k]}$ and $\tilde\sigma|_{[0,\tilde t]}$. 
Then $\tilde y$ is in $\omega^x_{\bar\s}\subset K(x)$.

Finally, using~\eqref{eq: invariance 1} with $y_1\in K(x)$ and the forward invariance of $K(x)$ for the dynamics of $\Sigma_x$, we have 
\begin{equation*}
\dist(y_0(t),K(x))\leq |y_0(t)-y_1(t)|\leq ce^{-\delta t}|y_0-y_1|, \quad \forall\, t\geq 0,
\end{equation*}
and one gets \eqref{eq:invariance} by arbitrariness of $y_1\in K(x)$.
 \end{proof}
 
\begin{remark}
The contents of the above proposition are essentially
contained in the preprint \cite{DellaRossa2022} (Theorem~1 and Proposition~2), where the authors study general switching affine  systems and the role of $K(x)$ is played by the set $\mathcal{K}_\infty$.
\end{remark}

\begin{lemma}\label{lem:KLiphom}
The set-valued map $K$ is globally Lipschitz continuous 
for the Hausdorff distance and homogeneous of degree one. 
\end{lemma}

\begin{proof}
Notice that, given $x\in \R^n$ and a nonzero $\lambda\in \R$, $y(\cdot)$ is a trajectory of $\Sigma_x$ if and only if $\lambda y(\cdot)$ is a trajectory of $\Sigma_{\lambda x}$. We deduce that $z\in \omega_{\s}^x$ if and only if $\lambda z\in \omega_{\s}^{\lambda x}$ and hence that $K(\lambda x) = \lambda K(x)$.


As for  the Lipschitz continuity of $K$, let $x_1, x_2\in \R^n$. It is easy to 
{deduce from the variation of constant formula} that there exists 
$L_K>0$
 independent of $x_1, x_2$ such that 
\begin{equation*}
|y_1(t)-y_2(t)|\leq L_K|x_1-x_2|, \qquad \forall\, t\geq 0,\forall\, \s\in\S_{\M},
\end{equation*}
where $y_1(\cdot)$ and $y_2(\cdot)$ are the trajectories of $\Sigma_{x_1}$ and $\Sigma_{x_2}$, respectively, associated with $\s$ and starting from the same initial condition $y_0$. By consequence, we have 
 \begin{equation*}
\max\{d(y_1, \omega_{\s}^{x_2}), d(y_2, \omega_{\s}^{x_1})\}\leq L_K|x_1-x_2|, 
\end{equation*} 
for $y_1\in\,\omega_{\s}^{x_1}$ and  $y_2\in\,\omega_{\s}^{x_2}$, where $\omega_{\s}^{x_1}$ and $\omega_{\s}^{x_2}$ are the $\omega$-limit sets of $\Sigma_{x_1}$ and $\Sigma_{x_2}$, respectively, which, thanks to Lemma~\ref{independance}, do not depend on the initial condition $y_0$. From the previous inequality  we obtain that, 
 \begin{equation*}
\max\{d(y_1, \cup_{\s}\omega_{\s}^{x_2}), d(y_2, \cup_{\s}\omega_{\s}^{x_1})\}\leq L_K|x_1-x_2|,
\end{equation*} 
for $y_1\in\,\cup_{\s}\omega_{\s}^{x_1}$ and  $y_2\in\,\cup_{\s}\omega_{\s}^{x_2}$.
By a standard density argument, 
 \begin{equation*}\label{eq2:lipschitz}
\max\{d(y_1, K(x_2)
), d(y_2,K(x_1) 
)\}
\leq L_K|x_1-x_2|,
\end{equation*} 
for $y_1\in K(x_1)$ and $y_2\in K(x_2)$, 
from which we obtain, by the arbitrariness of $y_1$ and $y_2$, that
\begin{equation*}
d_{H}(K(x_1), K(x_2))\leq L_K|x_1-x_2|,
\end{equation*}
where 
we recall that
$d_{H}$ denotes the Hausdorff distance in $\R^m$. 
\end{proof}

\subsection{
{
Asymptotic estimates by converse Lyapunov arguments
}}
The argument provided below bears similarities with proofs given in \cite{WATBLED2005362}, where more general dynamics are considered.

We consider the 
$\mu$-shifted differential inclusion 
\[\hat\Sigma_{\mu}: \quad \dot{\hat x}\in 
\{A_\mu\hat x+B y\mid M=\left(\begin{smallmatrix}A & B\\C & D\end{smallmatrix}\right)
\in \M,\; y\in K(\hat x)\},
\]
where 
$A_{\mu}=A+\mu I_{n}$ with $\mu\in\R$.
By homogeneity of $K(\cdot)$ it follows that for every solution $\hat x(\cdot)$ of $\hat\Sigma$ the trajectory $t\mapsto \hat{x}_\mu(t) = e^{\mu t}\hat{x}(t)$ is a solution of $\hat\Sigma_{\mu}$. As a consequence $\lambda(\hat\Sigma_{\mu}) = \lambda(\hat\Sigma)+\mu$. 
Hence, recalling that $\lambda(\Sigma_{\e,\mu}) = \lambda(\Sigma_\e)+\mu$,
in order to prove the right-hand side of inequality~\eqref{eigenvalues} it is enough to show that, under Assumption~\ref{UGES-fast}, $\lambda(\hat\Sigma_{\mu})<0$ implies $\lambda(\Sigma_{\e,\mu})<0$ for $\e$ small enough.
In order to prove the latter statement we will construct a common Lyapunov function for the systems $\Sigma_{\e,\mu}$ as the sum of Lyapunov functions for $\hat\Sigma_{\mu}$ and for $\Sigma_x$ defined next.
Assume that $\lambda({\hat\Sigma_\mu})<0$. Then, by Lemma~\ref{lem:ES-DI}, $|\hat x(t)|\leq ce^{-\gamma t}|\hat x(0)|$ for some $c\geq 1$ and $\gamma>0$, for every trajectory $\hat x(\cdot)$ of $\hat\Sigma_{\mu}$.
%
Define
\begin{equation}\label{eq:V1}
    V_1(x) = \sup_{\hat x(\cdot), t\in [0,\hat t\,]}e^{\gamma t}|\hat x(t)|^2,
\end{equation}
where the supremum is computed among all trajectories $\hat x(\cdot)$  of $\hat\Sigma_{\mu}$ starting from $x\in \R^n$, and $\hat t=\frac{\log(c)}{2\gamma}$.
%

Furthermore, let $\delta\in (0,|\lambda(\Sigma_D)|)$ and $\bar c\ge 1$ be such that $|y(t)|\leq \bar c e^{-\delta t}|y(0)|$ for every trajectory $y(\cdot)$ of $\Sigma_D$. By Proposition~\ref{forward invariance} one has 
\[\dist(y_x(t),K(x))\leq \bar c e^{-\delta t}\dist(y_x(0),K(x))\] for every trajectory $y_x$ of $\Sigma_x$, for every $x\in\R^n$ and $t>0$. Define
\begin{equation}\label{eq:V2}
    V_2(x,y) = \sup_{y_x(\cdot), t\in [0,\bar t\,]}e^{\delta t}\dist(y_x(t),K(x))^2,
\end{equation}
where the supremum is computed among all trajectories $y_x(\cdot)$ of $\Sigma_x$ starting from $y\in \R^m$  and $\bar t = \frac{\log(\bar c)}{2\delta}$.

In the next lemma, we summarize the main properties of $V_1$ and $V_2$.
\begin{lemma}\label{lem:V1V2}
The positive definite functions $V_1$ and $V_2$ introduced 
in \eqref{eq:V1} and \eqref{eq:V2} are homogeneous of degree two and locally Lipschitz continuous. Moreover, $V_1$ and $V_2$  are nonincreasing along every trajectory of $\hat\Sigma_\mu$ and $\Sigma_x$ respectively and satisfy the following estimates: 
\begin{align}
V_1(\hat x(t))&\leq e^{-\gamma t} V_1(\hat x(0)),\label{eq:LyapV1}\\
V_2(x,\bar{y}_x( t))&\leq e^{-\delta t} V_2(x,\bar{y}_x(0)),\label{eq:LyapV2}
\end{align}
where $t\geq 0$, $x\in \R^n$, $\hat x(\cdot)$ is an arbitrary trajectory of $\hat\Sigma_\mu$, and $\bar{y}_x(\cdot)$ is an arbitrary trajectory of $\Sigma_x$.
\end{lemma}
\begin{proof}
It is clear that both $V_1$ and $V_2$ are homogeneous of degree two. We give the proof of the remaining properties for $V_2$, the corresponding arguments for $V_1$ being completely analogous. 

Let us next show that $V_2$ is locally Lipschitz continuous.
For every bounded set $\mathcal{B}\subset \R^m$ there exists a compact set of $\R^m$ containing every trajectory of $\Sigma_x$ starting from $\mathcal{B}$. Take two points $y_1,y_2$ in $\mathcal{B}$ and consider the trajectories $y^1_x,y^2_{x}$ of $\Sigma_x$ corresponding to the same switching law starting respectively from $y_1$ and $y_2$. Then  $|y^2_x(t)-y^1_x(t)|\leq \bar c e^{-\delta t}|y_2-y_1|$  for every $t\geq 0$. We deduce that the function $y\mapsto e^{\delta t}\dist(y_x(t),K(x))^2$, where $y_x$ is the trajectory starting from $y\in \mathcal{B}$ corresponding to a fixed switching law $\sigma\in \S_\M$, is Lipschitz continuous, and the Lipschitz constant does not depend on $t$ nor on the switching law. Since the supremum among a family of uniformly Lipschitz continuous functions is Lipschitz continuous, we deduce that $V_2$ is locally Lipschitz continuous in the variable $y$. 
Similarly, local Lipschitz continuity  of the map $x\mapsto e^{\delta t}\dist(y_x(t),K(x))^2$ for a fixed switching law follows from the fact that $y_x(\cdot)$ is affine with respect to the variable $x$ and $K(\cdot)$ is Lipschitz continuous. Furthermore, the corresponding Lipschitz constant (locally) does not depend on $t$ nor on the switching law. 

Consider now a trajectory $\bar{y}_x(\cdot)$ of $\Sigma_x$ and let us prove that $V_2(x,\cdot)$ is nonincreasing along it. For $h>0$ one has 
\begin{align*}
V_2(x,\bar{y}_x( h)) & =  \sup_{\substack{y_x(\cdot),\ y_x|_{[0, h]}=\bar y_x|_{[0, h]}\\ t\in [ h,\bar t+ h]}}e^{\delta (t- h)}\dist(y_x(t),K(x))^2\\
& \leq  \sup_{\substack{y_x(\cdot),\ y_x(0)=\bar{y}_x(0)\\t\in [ h,\bar t+ h]}}e^{\delta (t- h)}\dist(y_x(t),K(x))^2\\
& \leq  \sup_{\substack{y_x(\cdot),\ y_x(0)=\bar{y}_x(0)\\t\in [0,\bar t+ h]}}e^{\delta (t- h)}\dist(y_x(t),K(x))^2\\
& =  e^{-\delta h}\sup_{\substack{y_x(\cdot),\ y_x(0)=\bar{y}_x(0)\\t\in [0,\bar t]}}e^{\delta t}\dist(y_x(t),K(x))^2\\
& = e^{-\delta h} V_2(x,\bar{y}_x(0)),
\end{align*}
where  we have used the fact that 
\begin{align*}
e^{\delta t}\dist(y_x(t),K(x))^2
&\leq \bar c^2 e^{-\delta t}\dist(y_x(0),K(x))^2<\dist(y_x(0),K(x))^2
\end{align*} 
for every $t>\bar t$ along every trajectory of $\Sigma_x$.
%
\end{proof}

 Based on the above construction of $V_1$ and $V_2$, we  next show the existence of a common Lyapunov function allowing us to prove that  $(\Sigma_{\e,\mu})_{\e>0}$ is $\e$-ES. \\
\begin{proposition}
There exists $\chi>0$, $0 < \alpha_- < \alpha_+$, $\eta>0$, and $\e_*>0$ such that, setting $V = V_1+\chi V_2$, one has
\begin{align}
 \alpha_-|(x,y)|^2 \leq V(x,y) \leq \alpha_+|(x,y)|^2\label{Lyap-1},\qquad &\forall x\in \R^n,\;\forall y\in \R^m,\\
 V(x(t),y(t))\leq V(x(0),y(0)) e^{-\eta t},\qquad&\forall t\ge 0,\label{Lyap-2}
\end{align}
where~\eqref{Lyap-2} holds true for every solution $(x(\cdot),y(\cdot))$  of $\Sigma_{\e,\mu}$ for $\e<\e_*$. As a consequence $\limsup_{\e\to 0^+} \lambda(\Sigma_{\e}) \leq \lambda(\hat\Sigma)$ and if $\lambda(\hat\Sigma)<0$  then $\Sigma$ is $\e$-ES.
\end{proposition}

\begin{proof}
The right inequality in~\eqref{Lyap-1} follows from the bounds 
\begin{align*}
V_1(x)\leq c^2 |x|^2
\end{align*}
and
\begin{align*}
V_2(x,y)\leq \bar c^2\, \dist(y,K(x))^2&\leq \bar c^2\, (|y|+\dist(0,K(x)))^2
\leq \bar c^2(|y|+L_K|x|)^2,
\end{align*}
where $L_K$ is the Lipschitz constant for $K(\cdot)$. Concerning the left inequality in~\eqref{Lyap-1}, note that $\dist(y,K(x))\geq |y|-L_K|x|$, where $L_K$ is the Lipschitz constant for $K(\cdot)$. Then, either $|y|> 2 L_K|x|$, in which case $\dist(y,K(x))> \frac12 |y|$ and 
\[V(x,y)\geq |x|^2+\chi \dist(y,K(x))^2>|x|^2+\frac{\chi}4|y|^2, \]
or $|y|\leq 2 L_K|x|$, in which case 
\[V(x,y)\geq |x|^2\geq \frac{1}{1+4L_K^2} |(x,y)|^2.\]
The desired inequality holds true with \[\alpha_-=\min\left\{\frac{1}{1+4L_K^2},\frac{\chi}4\right\}.\]
Consider now a trajectory $(x(\cdot),y(\cdot))$  of $\Sigma_{\e,\mu}$ corresponding to the switching law $\s$. In order to prove~\eqref{Lyap-2} we first estimate the difference $V(x(\e h),y(\e h))-V(x(0),y(0))$ for small $h$. 
Note that 
\begin{align*}
&|x(\e h)-x(0)|=\e|(x(0),y(0))|\, O(h),\\
&|y(\e h)-y(0)|=|(x(0),y(0))|\,O(h),
\end{align*}
where $O(h)$ denotes a function bounded in absolute value by $Ch$, 
where the constant $C$ does not depend on $x(0)$, $y(0)$, $h$, nor $\e$, for $\e,h$ in a small right-neighborhood of zero.  
%
%
%
%
Knowing that \[t\mapsto\dist(\dot x(t),A_{\mu}(t)x(t)+B(t)K(x(t)))\leq \|B\|\dist(y(t),K(x(t)))\] is integrable, we deduce from  
Theorem~10.4.1 in \cite
{Aubin-Frankowska} that
there exists a solution $\tilde x(\cdot)$ of $\hat\Sigma_{\mu}$ 
such that $\tilde x(0)=x(0)$ and 
\begin{align*}\label{eq:paolo}
|x(\e h)-\tilde x(\e h)|
&\leq \|B\|e^{(\|A_{\mu}\|+L_K\|B\|)\e h}\int_0^{\e h}{\dist(y(s),K(x(s)))ds}\\
& \leq C_1\e h \,\dist(y(0),K(x(0)))+\e|(x(0),y(0))| \,O(h^2),
\nonumber
\end{align*}
for some $C_1>0$, where we have used the fact that
\[\dist(p,P)\leq |p-q| + \dist(q,Q) + d_H(Q,P), \]
for $p,q\in \R^m$ and $P,Q\subset \R^m$.
Hence 
\begin{align*}
V_1(x(\e h)) - V_1(\tilde x( h))\leq& \e |(x(0),y(0))|\,\dist(y(0),K(x(0)))\,O( h)
+\e |(x(0),y(0))|^2\,O( h^2),
\end{align*}
as it follows from the fact that the Lipschitz constant of $V_1$ on a ball of radius $r$ is of order $r$. 
By using~\eqref{eq:LyapV1},
\begin{align*}
V_1(x(\e h)) - V_1(x(0))
&=V_1(x(\e h)) - V_1(\tilde x(\e h)) + V_1(\tilde x(\e h)) -  V_1(x(0))\\
&\leq V_1(x(0))(e^{-\gamma \e h}-1)+\e|(x(0),y(0))|^2\,O( h^2) +\e |(x(0),y(0))|\,\dist(y(0),K(x(0)))\,O( h).
\end{align*}
Moreover, 
\begin{align*}
V_2(x(\e h),y(\e h))-V_2(x(0),y(0)) 
&\leq  |V_2(x(\e h),y(\e h))- V_2(x(0),y(\e h))| 
+ |V_2(x(0),y(\e h))-V_2(x(0),\tilde y(\e h))|\\
&+ V_2(x(0),\tilde y(\e h)) -V_2(x(0),y(0)),
\end{align*}
where $\tilde y({\cdot}/{\e})$ is the solution of $\Sigma_{x(0)}$ starting at $y(0)$
and corresponding to the signal $\sigma({\cdot}/{\e})$. 
The first term is of order  $\e |(x(0),y(0))|^2 O(h)$, while, by~\eqref{eq:LyapV2}, \[V_2(x(0),\tilde y(\e h)) -V_2(x(0),y(0))\leq (e^{-\delta h}-1)V_2(x(0),y(0)).\] Furthermore 
\begin{align*}
\frac{d}{dt}(y(t)-\tilde y(t)) &= \frac1{\e}D(t)(y(t)-\tilde y(t)) 
+ \frac1{\e}C(t)(x(t)-x(0))+\mu y(t),
\end{align*}
from which one gets that \[|y(\e h)-\tilde y(\e h)| =\e |(x(0),y(0))|\,O( h),\] so that 
\[
|V_2(x(0),y(\e h))-V_2(x(0),\tilde y(\e h))|=\e|(x(0),y(0))|^2\,O( h).
\]
Summing up, 
\begin{align*}
&V_2(x(\e h),y(\e h))-V_2(x(0),y(0))
\leq (e^{-\delta h}-1)V_2(x(0),y(0)) + \e|(x(0),y(0))|^2\,O(h)
\end{align*}
and
\begin{align*}
V(x(\e h),y(\e h))  -V(x(0),y(0))  
&\leq  (e^{-\gamma \e h}-1)V_1(x(0)) + \chi(e^{-\delta h}-1)V_2(x(0),y(0))\\
&+ C_2 \e h|(x(0),y(0))|\,\dist(y(0),K(x(0)))  \\
&+ C_3 \chi \e h|(x(0),y(0))|^2+ C_4 \e h^2|(x(0),y(0))|^2,
\end{align*}
where  $C_2,C_3,C_4$ do not depend on $x(0)$, $y(0)$, $h$, nor $\e$, for $h,\e$ in a small right-neighborhood of zero.  
We have 
\begin{align*}
\limsup_{h\to 0^+}\frac{V(x(\e h),y(\e h))  -V(x(0),y(0))}{h}
&\leq -\gamma\e V_1(x(0)) -  \delta\chi V_2(x(0),y(0))
+ C_2 \e |(x(0),y(0))|\,\dist(y(0),K(x(0)))\\
& + C_3 \chi \e |(x(0),y(0))|^2.
\end{align*}
We write the right-hand side as the sum of the three terms 

\begin{align*}
W_1 &= -\frac{\gamma\e}2 (V_1(x(0)) + \chi V_2(x(0),y(0)))
=-\frac{\gamma\e}2 V(x(0),y(0)),\\
W_2 &= -\frac{\chi}4 (2\delta-\gamma\e)V_2(x(0),y(0)) - C_3  \chi \e |(x(0),y(0))|^2 -C_2\e |(x(0),y(0))|\,\dist(y(0),K(x(0))) \\
&\leq  -\frac{\chi}4 (2\delta-\gamma\e)\dist(y(0),K(x(0))^2 - C_3  \chi \e |(x(0),y(0))|^2 -C_2\e |(x(0),y(0))|\,\dist(y(0),K(x(0))), \\
W_3 &= -\frac{\gamma\e}2 V(x(0)) -\frac{\chi}4 (2\delta-\gamma\e)V_2(x(0),y(0)) + 2C_3  \chi \e |(x(0),y(0))|^2\\
&\leq -\frac{\gamma\e}2 |x(0)|^2 -\frac{\chi}4 (2\delta-\gamma\e)\dist(y(0),K(x(0))^2 + 2C_3  \chi \e |(x(0),y(0))|^2,
\end{align*}
where the inequalities in $W_2,W_3$ are obtained assuming $\e\leq \frac{2\delta}{\gamma}$.
For any given $\chi>0$, it is easy to see that $W_2\leq 0$ if $\e$ is small enough.
Since \[|y(0)|\leq L_K|x(0)|+\dist(y(0),K(x(0)))\] we get 
\[|(x(0),y(0))|^2 \leq (1+2L_K^2)|x(0)|^2 + 2\dist(y(0),K(x(0)))^2\] so that $W_3\leq 0$, provided that $\chi$ is chosen so that 
\[-\frac{\gamma}2 + 2C_3  \chi (1+2L_K^2) \leq 0\] 
and $\e$ is small enough.
By a time-shift we obtain, for $t\geq 0$,
\begin{align*}
    \limsup_{\tau\to 0^+}&\frac{V(x(t+\tau),y(t+\tau))  -V(x(t),y(t))}{\tau}
\leq -\frac{\gamma}2 V(x(t),y(t)).
\end{align*}
Since $V(x(\cdot),y(\cdot))$ is absolutely continuous we deduce that 
\[\frac{d}{dt}V(x(t),y(t)) \leq -\frac{\gamma}{2} V(x(t),y(t)), \quad \hbox{a.e. }t\geq 0,\] 
and 
\[V(x(t),y(t))\leq V(x(0),y(0)) e^{-\frac{\gamma}{2}t}, \quad \forall\, t\geq 0,\] concluding the proof.
\end{proof}

\bibliography{biblio}           

\end{document}